\documentclass[nospthms,reqno]{svjour2}

\usepackage{amssymb, tikz}
\usepackage{mathrsfs}
\usepackage{dsfont}
\usepackage{amsfonts,enumitem, enumerate, xspace, amsthm}
\usepackage{changebar}
\usepackage{hyperref}
\usepackage{pdfsync}
\usepackage{color}
\usepackage{bm}
\usepackage{amsmath}

\numberwithin{equation}{section}

\newtheorem{thm}{Theorem}[section]
\newtheorem{lem}[thm]{Lemma}
\newtheorem{cor}[thm]{Corollary}

\newtheorem{prop}[thm]{Proposition}

\newtheorem{defin}{Definition}[section]

\newcommand\eps{\epsilon}

\def\om{{\omega}}
\def\Om{{\Omega}}
\def\dsum{\displaystyle\sum}
\def\dlim{\displaystyle\lim}
\def\dsup{\displaystyle\sup}
\def\dint{\displaystyle\int}
\def\Erfc{\mathrm{Erfc}}
\allowdisplaybreaks

\begin{document}

\title{Particle model for the reservoirs in the simple symmetric exclusion process}

\author{Thu Dang Thien Nguyen}

\institute{Thu Dang Thien Nguyen \at
            Gran Sasso Science Institute, Viale Francesco Crispi, 7, L'Aquila  67100, Italy\\
			 Department of Mathematics, University of Quynhon, Quy Nhon, Vietnam\\
             \email{thu.nguyen@gssi.it}           
}

\date{Received: \today / Accepted: date}

\maketitle

\begin{abstract}
{In this paper, we will study the long time behavior of the simple symmetric exclusion process in the ``channel'' $\Lambda_N=[1,N]\cap \mathbb{N}$ with reservoirs at the boundaries. These reservoirs are also systems of particles which can be exchanged with the particles in the channel. The size $M$ of each reservoir is much larger than the one of the channel, i.e.  $M=N^{1+\alpha}$ for a fixed number $\alpha>0$. Based on the size of the channel and the holding time at each reservoir, we will investigate some types of rescaling time.
}
\end{abstract}
\maketitle

\section{Introduction}
Let us recollect some results on the most studied model of type of open systems, the simple symmetric exclusion process (SSEP) in an interval of size $N$, complemented by birth-death processes at the boundaries. Such birth-death processes are called ideal reservoirs of fixed densities $v_{0,\pm}$. As a result of the duality technique, we can obtain the following well-known expression which is extremely useful in order to investigate the macroscopic behavior of our particle system.
\begin{align}\label{g}
&\mathcal{E}^N\big[\eta(x,t)\big]\notag\\
=&E_x\big[u_0(N^{-1} Y(t))\mathbf{1}_{\tau^Y_0\wedge \tau_{N+1}^Y>t}\big]+v_{0,-}P_x\big(\tau^Y_0<t, \tau^Y_0<\tau^Y_{N+1}\big)\notag\\
&\hskip 4cm+v_{0,+}P_x\big(\tau^Y_{N+1}<t, \tau^Y_0>\tau^Y_{N+1}\big),
\end{align}
where $\mathcal{E}^N$ stands for the expectation with respect to our process, $P_x$ and $E_x$ denote the law and expectation of a simple symmetric random walk $Y$ on $\mathbb{Z}$ starting from $x$, $\tau_a^Y$ is the first time that the random walk $Y$ hits $a$ and $u_0\in C(0,1)$ represents the initial macroscopic density of particles. Namely, when the dynamics is speeded up in time by $N^2$, in the limit the above expression will give us the evolution equation which is described by the linear heat equation on $[0,1]$ with given boundary conditions. As the time goes faster than the order $N^2$, i.e. $t/N^2\to \infty$, it can be derived that the particle density is close to the linear interpolation between $v_{0.-}$ and $v_{0,+}$. Moreover, the propagation of chaos property holds for this model, it means that the number of particles at macroscopic distances become independent. The arguments for the proof can be found in \cite{GKMP}.

In the current article, we will regard the reservoirs as particle models of $N^{1+\alpha}$ sites for a fixed number $\alpha>0$. These particle systems interact with the SSEP by exchanging particles in a mean field fashion (see the next section for more details). Then under a suitable assumption on the initial state, when the time is scaled by $N^{2+\alpha'}$ for $\alpha'\in [0,\alpha)$, we will prove that in the limit $N\to\infty$, the particle system in the channel evolves like the SSEP in an interval of size $N$, imposed by ideal reservoirs of fixed densities at the extremes. A rigorous analysis to attain the ideal reservoir limit and the propagation of chaos property for our present particle system can be found in Section \ref{Ideal}.

In Section \ref{Adiabatic}, we will get the so-called adiabatic limit. More precisely, in the time scale $N^{2+\alpha}t$, the particle's density converges, as $N\to\infty$, to the linear interpolation between $v_-(t)$ and $v_+(t)$, where $v_-(t)$ and $v_+(t)$ changing in time are respectively the density of particles at the boundaries. The adiabatic limit can also be obtained when one considers the case where the boundary densities depend directly on time in the context of the SSEP, see \cite{DO}.

When the time scale is of order $N^{2+\alpha'}$ for $\alpha'>\alpha$, one can verify that in the global equilibrium limit, the particle density converges to the average of the initial densities $v_{0,-}$ and $v_{0,+}$ at the boundaries. Section \ref{Global} is devoted to this verification.

The propagation of chaos property for our present model has not been extended to the case when $\alpha'\geq\alpha$, hence it is still an open problem so far.

We also notice that if our particle system evolves as an SSEP on the semi-infinite domain $[0,\infty)$ with a reservoir at the origin then only one relevant time rescaling is studied. In the hydrodynamic limit, such system evolves according to the linear heat equation on $[0,\infty)$ whose solution can be represented probabilistically in terms of the Brownian motion sticky at the origin. The verification can be done by extending the result given by Amir in \cite{A} to the convergence of continuous time sticky random walks.

In \cite{N}, we consider the case $\alpha=0$ where each reservoir and the channel have the same size. If the time is scaled by $N^2$ while the space is scaled down by $N^{-1}$, we can obtain the hydrodynamic limit for such system as $N\to\infty$. That is a linear heat equation on $(0,1)$ with free boundary conditions.

\section{Model and sticky random walk}
\subsection{The model}
We study a simple symmetric exclusion process in the interval $\Lambda_N=[1,N]\cap \mathbb{N}$ which we will call ``the channel'', $N$ being a positive integer. It is complemented by reservoirs $\mathcal{S}_{\pm}$ of size $M\in\mathbb{N}^*$ at the extremes. Physically, the reservoirs are also particle systems which interact with our given system in the channel by exchanging particles. Thus the state space is $\Om^*=\{0,1\}^{\Lambda_N} \times \{0,1\}^{\mathcal S_-} \times \{0,1\}^{\mathcal S_+}$: the elements $\om$ of $\Om^*$ are sequences $\eta(x)$, $x \in \Lambda_N \cup \mathcal S_-\cup \mathcal S_+$, where $\eta(x)\in \{0,1\}$ is understood as the number of particles at site $x$, and the generator is
 \begin{align}
 \mathbf{L}^*f(\om) =&\frac 12 \sum_{x=1}^{N-1}[ f(\om^{x,x+1})-f(\om)]\notag\\
 &+\frac 1{2M} \sum_{z \in \mathcal S_-}[ f(\om^{1,z})-f(\om)]+\frac 1{2M} \sum_{z \in \mathcal S_+}[ f(\om^{N,z})-f(\om)]\notag\\
&+\sum_{x,y\in\mathcal{S}_+}c(x,y)[ f(\om^{x,y})-f(\om)]+\sum_{x,y\in\mathcal{S}_-}c(x,y)[ f(\om^{x,y})-f(\om)], \label{2.4}
 \end{align}
where $\om^{x,y}$ is obtained from $\om$ by exchanging $\eta(x)$ and $\eta(y)$; $c(x,y)$ is the jump rate from $x$ to $y$.

We denote by $n_{\pm}$ the number of particles in $\mathcal S_{\pm}$, respectively. Then one can see that the marginal over the variables $\eta(x), x \in \Lambda_N$, and $n_{\pm}$ has the law of the Markov process with generator $\mathbf{L}$ given by
 \begin{equation}
 \label{2.1}
 \mathbf{L}f(\om) =\frac 12 \sum_{x=1}^{N-1}[ f(\om^{x,x+1})-f(\om)]
 + c_N(\om)[ f(\om^N)-f(\om)] +  c_{1}(\om)[ f(\om^{1})-f(\om)]
 \end{equation}
where an element $\om$ of $\Om=\{0,1\}^{\Lambda_N} \times \{0,..,M\} \times \{0,..,M\}$ is a sequence $\eta(x)$, $x=1,..,N$, complemented by two integers $n_-$ and $n_+$ both in $[0,M]$;
 \begin{equation}
 \label{2.2}
 c_N(\om) = \frac 12\Big( 1-\frac{n_+}{M}\Big)\eta(N)+\frac 12\,\frac{n_+}{M}(1-\eta(N));
 \end{equation}
$\om^N$ has $\eta^N(N)=0$ and $n_+^N= n_++1$ if $\eta(N)=1$; and $\om^N$ has $\eta^N(N)=1$ and $n_+^N= n_+-1$ if $\eta(N)=0$. $c_1(\om)$ and $\om^1$ are defined similarly.
\subsection{Sticky random walk}\label{regularity}
In the sequel we identify $\eps=1/N$. Denote $\bar{\Lambda}_N:=[0,N+1]\cap\mathbb{N}$. Let us fix $\alpha>0$ and set hereafter $M=N^{1+\alpha}$. We write for $x\in\bar{\Lambda}_N$,
$$ \rho^{\eps}(x,t) = \mathbf{E}^\eps\Big[ \eta(x,t)\Big],$$
where $\eta(0,t):=\frac{n_-(t)}{M}, \,\eta(N+1,t):=\frac{n_+(t)}{M}$ and $\mathbf{E}^\eps$ stands for the expectation with respect to the process with generator \eqref{2.1}, the initial distribution will be specified later. Sometimes we use $\rho^{\eps}_0(x)$ instead of $\rho^{\eps}(x,0)$.

We denote $\rho^{\eps}_-(t):=\rho^{\eps}(0,t),\quad  \rho^{\eps}_+(t):=\rho^{\eps}(N+1,t)$.

Then for $x\in \Lambda_N$, we have
\begin{equation} \label{2.6}
 \frac{d}{dt}\rho^{\eps}(x,t) =  \frac 12 \Big( \rho^{\eps}(x-1,t)+\rho^{\eps}(x+1,t)
 -2\rho^{\eps}(x,t)\Big).
 \end{equation}

 We also have  
\begin{equation}
 \label{2.7}
 \frac{d}{dt}\rho^{\eps}_-(t) =  \frac 1{2M} \Big( \rho^{\eps}(1,t) -\rho^{\eps}_-(t)\Big), \quad \frac{d}{dt}\rho^{\eps}_+(t) =  \frac 1{2M} \Big( \rho^{\eps}(N,t)-\rho^{\eps}_+(t)\Big)
 \end{equation}

Moreover, there is a close relationship between our particle model and a sticky random walk.
\begin{defin}\label{SRW}
Sticky random walk $(X(t))_{t\geq 0}$ moving on $\bar{\Lambda}_N$ is a continuous time random walk with jump rates $c(x,x\pm 1)=\dfrac{1}{2}, \forall x\in\Lambda_N$ and $c(0,1)=c(N+1,N)=\dfrac{1}{2M}$.

More precisely, the generator $L$ of the random walk $X$ is given by
\begin{align}
Lf(x)=&\mathbf{1}_{1\leq x\leq N}\,\dfrac{1}{2}[f(x+1)+f(x-1)-2f(x)]\notag\\
+&\mathbf{1}_{x=0}\,\dfrac{1}{2M}[f(x+1)-f(x)]+\mathbf{1}_{x=N+1}\,\dfrac{1}{2M}[f(x-1)-f(x)].\label{gen}
\end{align}
\end{defin}

We call $E_x$ and $p_t(\cdot,\cdot)$ the expectation and the transition density of the random walk $X$ starting from $x$. Then by applying the same argument as in Proposition 3.1 and Corollary 3.2, \cite{N}, it is easy to check the following result which is similar to the expression \eqref{g}.
\begin{prop}\label{prop3.1}
For any $x\in \bar{\Lambda}_N$ and $t\geq 0$,
\begin{equation}\label{duality*}
\rho^{\eps}(x,t)=E_x\big[\rho^{\eps}(X(t),0)\big]=E_x\Big[\rho^{\eps}(X(t),0)\mathbf{1}_{\tau>t}+\rho^{\eps}(X(\tau),t-\tau)\mathbf{1}_{\tau\leq t}\Big],
\end{equation}
where $\tau$ is the first hitting time of $X$ into $\{0,N+1\}$.
\end{prop}
\section{Ideal reservoir limit}\label{Ideal}
\begin{thm}[Ideal reservoir limit]\label{Idea}
Let $u_0 \in C(0,1)$ with values in $[0,1]$ and $v_{0,\pm}\in [0,1]$. Suppose that at the initial time,
\begin{equation}\label{e}
\dfrac{n_{\pm}(0)}{M}\xrightarrow{\mathbf{P}^{\eps}}v_{0,\pm},\quad \rho^{\eps}(x,0) = u_0(\eps x), \forall x\in\Lambda_N.
 \end{equation}

Then for any $\alpha'\in [0,\alpha)$ and $T>0$, 
\begin{equation}\label{f}
\dlim_{N\to\infty} \sup_{x\in\Lambda_N}\sup_{t\in [0,T]}\big|\mathbf{E}^{\eps}\big[\eta(x,N^{2+\alpha'}t)\big]-\mathcal{E}^{\eps}\big[\eta(x,N^{2+\alpha'}t)\big]\big|=0,
\end{equation}
where $\mathcal{E}^{\eps}$ denotes the expectation with respect to the simple symmetric exclusion process in $\{0,1\}^{\Lambda_N}$ with reservoirs at the boundaries of fixed densities $v_{0,\pm}$. 
\end{thm}
\begin{proof}
Let us first verify the following result.
\begin{lem}\label{fixed}
For any $\delta>0, \alpha'\in [0,\alpha)$ and $T>0$,
$$ \dlim_{N\to\infty} \mathbf{P}^{\eps}\Big(\sup_{t\in [0,T]}\Big|\dfrac{n_{\pm}(N^{2+\alpha'}t)-n_{\pm}(0)}{M}\Big|\leq \delta\Big)=1. $$
\end{lem}
\begin{proof}
We define $\sharp(\mathcal{S}_+\to\mathcal{S}_-)$ as the number of particles which go from the right reservoir $\mathcal{S}_+$ to the left one $\mathcal{S}_-$ from the initial time to $N^{2+\alpha'}t$. Thus we have the estimate
$$ n_-(N^{2+\alpha'}t)\leq n_-(0)+N+\sharp(\mathcal{S}_+\to\mathcal{S}_-). $$

Then
$$ \dfrac{n_-(N^{2+\alpha'}t)- n_-(0)}{M} \leq \dfrac{N}{N^{1+\alpha}}+\dfrac{\sharp(\mathcal{S}_+\to\mathcal{S}_-)}{M}.$$

Let us estimate $p_t(N+1,0)$ the probability that the sticky random walk $X$ starting from $N+1$ reaches $0$ at time $t$. We call $\nu_t$ the number of times that $X$ reaches $N+1$ up to time $t$ and $\{\sigma^{(k)}_{N+1}, k\geq 1\}$ the successive return times to $N+1$ of the sticky random walk $X$ starting from $N+1$, i.e.
\begin{align*}
&\sigma^{(1)}_{N+1}=0\\
&\kappa^{(1)}_N=\inf\{s>\sigma^{(1)}_{N+1}, X(s)=N\}\\
&\sigma^{(2)}_{N+1}=\inf\{s>\kappa^{(1)}_N, X(s)=N+1\}\\
&\ldots
\end{align*}

Then we can write
\begin{align}
&p_t(N+1,0)=\dint_0^tP_{N+1}(X(t)=0, \tau_{N+1}^X<s, \tau_0^X=s)\,ds\notag\\
\leq&\dint_0^t\dsum_{n\geq 1}P_{N+1}(X(t)=0, \sigma^{(n)}_{N+1}<s<\sigma^{(n+1)}_{N+1}, \tau_0^X=s)\,ds\notag\\
\leq&\dint_0^t\dsum_{n\geq 1}\dint_0^sP_{N+1}(\sigma^{(n)}_{N+1}=\theta)P_{N+1}(X(t-\theta)=0, \tau_0^X<\sigma^{(2)}_{N+1})\,d\theta\,ds\notag\\
\leq&\dfrac{t}{2M}\dint_0^t\dint_0^sP_{N+1}(X(t-\theta)=0, \tau_0^X<\sigma^{(2)}_{N+1})\,d\theta\,ds\notag\\
\leq&\dfrac{t}{2M}\dint_0^t\dint_0^s\dint_0^{t-\theta}P_{N+1}(\tau_N^X=\upsilon)P_N(X(t-\theta-\upsilon)=0, \tau_0^X<\tau^X_{N+1})\,d\upsilon\,d\theta\,ds\notag\\
\leq&\dfrac{1}{N}\dfrac{t}{2M},\label{num}
\end{align}
where the third inequality follows from the fact that
\begin{align*}
\dsum_{n\geq 1}P_{N+1}(\sigma^{(n)}_{N+1}=\theta)=&\dsum_{n\geq 1}\dsum_{m\geq n}P_{N+1}(\sigma^{(n)}_{N+1}=\theta|\nu_t=m)P_{N+1}(\nu_t=m)\\
\leq&\dsum_{m\geq 1}mP_{N+1}(\nu_t=m)=E\big[\nu_t\big]\leq \dfrac{t}{2M}.
\end{align*}

By \eqref{num}, we have $\dsup_{t\in [0,T]}p_{N^{2+\alpha'}t}(N+1,0)\to 0$ for any $\alpha'\in [0,\alpha)$ and $T>0$. So in probability, 
$$ \dlim_{N\to\infty}\sup_{t\in [0,T]}\dfrac{\sharp(\mathcal{S}_+\to\mathcal{S}_-)}{M}=0. $$

This gives us that for any $\delta>0$ and $T>0$,
$$ \dlim_{N\to\infty} \mathbf{P}^{\eps}\Big(\sup_{t\in [0,T]}\dfrac{n_-(N^{2+\alpha'}t)-n_-(0)}{M}\leq \delta\Big)=1. $$

On the other hand, we also have
$$ n_-(N^{2+\alpha'}t)\geq \sharp(\mathcal{S}_-\to \mathcal{S}_-). $$

Then it follows that
$$ \dfrac{n_-(N^{2+\alpha'}t)- n_-(0)}{M} \geq -\dfrac{\sharp(\mathcal{S}_-\to\mathcal{S}_+)+\sharp(\mathcal{S}_-\to \text{channel})}{M}.$$

By \eqref{num}, we get $\dsup_{t\in [0,T]}\dsum_{x=1}^{N+1}p_{N^{2+\alpha'}t}(0,x)\to 0$ for any $\alpha'\in [0,\alpha)$ and $T>0$. Hence, in probability, 
$$ \dlim_{N\to\infty}\sup_{t\in [0,T]}\dfrac{\sharp(\mathcal{S}_-\to\mathcal{S}_+)+\sharp(\mathcal{S}_-\to \text{channel})}{M}=0. $$

It implies that for any $\delta>0$ and $T>0$,
$$ \dlim_{N\to\infty} \mathbf{P}^{\eps}\Big(\sup_{t\in [0,T]}\Big(-\dfrac{n_-(N^{2+\alpha'}t)-n_-(0)}{M}\Big)\leq \delta\Big)=1. $$
\end{proof}

As a consequence of Lemma \ref{fixed}, we obtain that for any $\alpha'\in [0,\alpha)$ and $T>0$,
\begin{equation}\label{res}
\dlim_{N\to\infty}\sup_{t\in [0,T]}\big|\rho^{\eps}_{\pm}(N^{2+\alpha'}t)-v_{0,\pm}\big|=0. 
\end{equation}

Since for $x\in \Lambda_N$, 
\begin{align*}
F_{N,x}(s)&=P_x(\tau^X_0\leq s, \tau^X_0<\tau^X_{N+1})=P_x(\tau^Y_0\leq s, \tau^Y_0<\tau^Y_{N+1})=:\mathds{F}_{N,x}(s),\\
G_{N,x}(s)&=P_x(\tau^X_{N+1}\leq s, \tau^X_0>\tau^X_{N+1})=P_x(\tau^Y_{N+1}\leq s, \tau^Y_0>\tau^Y_{N+1})=:\mathds{G}_{N,x}(s),
\end{align*}
then by Proposition \ref{prop3.1}, we can attain that
\begin{align}\label{main1}
&\mathbf{E}^{\eps}\big[\eta(x,t)\big]\notag\\
=&E_x\big[u_0(\eps X(t))\mathbf{1}_{\tau>t}\big]+E_x\big[\rho^{\eps}(X(\tau),t-\tau)\mathbf{1}_{\tau\leq t}\big]\notag\\
=&E_x\big[u_0(\eps Y(t))\mathbf{1}_{\tau_0^Y\wedge\tau_{N+1}^Y>t}\big]+\dint_0^t\rho^{\eps}_-(t-s)\,dF_{N,x}(s)+\dint_0^t\rho^{\eps}_+(t-s)\,dG_{N,x}(s)\notag\\
=&E_x\big[u_0(\eps Y(t))\mathbf{1}_{\tau_0^Y\wedge\tau_{N+1}^Y>t}\big]+\dint_0^t\rho^{\eps}_-(t-s)\,d\mathds{F}_{N,x}(s)+\dint_0^t\rho^{\eps}_+(t-s)\,d\mathds{G}_{N,x}(s).
\end{align}

On the other hand, we know that for $x\in\Lambda_N$,
\begin{align}\label{fi}
&\mathcal{E}^{\eps}\big[\eta(x,t)\big]\notag\\
=&E_x\big[u_0(\eps Y(t))\mathbf{1}_{\tau_0^Y\wedge\tau_{N+1}^Y>t}\big]+\dint_0^tv_{0,-}\,d\mathds{F}_{N,x}(s)+\dint_0^tv_{0,+}\,d\mathds{G}_{N,x}(s).
\end{align}

Therefore, it can be deduced that
\begin{align*}
&\big|\mathbf{E}^{\eps}\big[\eta(x,N^{2+\alpha'}t)\big]-\mathcal{E}^{\eps}\big[\eta(x,N^{2+\alpha'}t)\big]\big|\\
\leq &\dint_0^t\big|\rho^{\eps}_-(N^{2+\alpha'}(t-s))-v_{0,-}\big|\,d\mathds{F}_{N,x}(N^{2+\alpha'}s)\\
&\quad+\dint_0^t\big|\rho^{\eps}_+(N^{2+\alpha'}(t-s))-v_{0,+}\big|\,d\mathds{G}_{N,x}(N^{2+\alpha'}s).
\end{align*}

This implies \eqref{f} by using \eqref{res}.
\end{proof}
\begin{cor}
\begin{itemize}
\item[a)] For any $r \in [0,1]$ and $t>0$, $\rho^\eps([Nr],N^2t)$ and  $\rho^\eps_{\pm}(N^2t)$ have limits $u(r,t)$ and respectively $v_{0,\pm}$, where
\begin{equation}\label{linearheat}
\begin{cases}
&u_t =  \frac 1{2} u_{rr},  \quad r\in [0,1]\\
&u(1,t)=v_{0,+}\\
&u(0,t)=v_{0,-}\\
&\dlim_{t\downarrow 0}u(r,t) =u_0(r), \quad \forall r\in (0,1).
\end{cases}
\end{equation}
\item[b)] For any $\alpha'\in (0,\alpha)$, $r \in [0,1]$ and $t>0$, the sequences $\rho^\eps([Nr],N^{2+\alpha'}t)$ and  $\rho^\eps_{\pm}(N^{2+\alpha'}t)$ have limits $u(r,t)$ and respectively $v_{0,\pm}$, where
$$u(r,t)=\big(v_{0,+}-v_{0,-}\big)r+v_{0,-}.$$
\end{itemize}
\end{cor}
\begin{proof} Since Part a) is a direct consequence of Theorem \ref{Idea}, it remains to check Part b). We observe that for any $x\in\Lambda_N$,
\begin{equation}\label{prob1}
P_x(\tau_0^Y\wedge\tau_{N+1}^Y>t)\leq \dfrac{Cx}{\sqrt{t}}
\end{equation}
then the first term on the right hand side of \eqref{fi} converges to $0$ as $N\to\infty$, for any $\alpha'\in (0,\alpha)$.

Moreover, the estimate \eqref{prob1} also allows us to obtain that
$$|\mathds{F}_{N,x}(N^{2+\alpha'}t)-P_x(\tau_0^Y<\tau_{N+1}^Y)|\leq P_x(\tau_0^Y>N^{2+\alpha'}t)\leq \dfrac{Cx}{N^{1+\frac{\alpha'}{2}}\sqrt{t}}.$$

Since
$$ P_x(\tau_0^Y<\tau_{N+1}^Y)=1-\dfrac{x}{N+1}, $$
it implies that for any $(r,t)\in [0,1]\times (0,\infty)$, 
$$ \dlim_{N\to\infty}\dint_0^{N^{2+\alpha'}t}v_{0,-}\,d\mathds{F}_{N,x}(s) = v_{0,-}(1-r). $$

Analogously, we also have
$$ \dlim_{N\to\infty}\dint_0^{N^{2+\alpha'}t}v_{0,+}\,d\mathds{G}_{N,x}(s) = v_{0,+}r. $$

Therefore, the sequence $ \rho^{\eps}([Nr], N^{2+\alpha'}t)$ converges pointwise on $[0,1]\times (0,\infty)$ to $u$, where
$$u(r,t)=\big(v_{0,+}-v_{0,-}\big)r+v_{0,-}.$$
\end{proof}

\begin{thm}[Propagation of chaos]\label{POCsm}
Assume that at the initial time, \eqref{e} is satisfied. Then for any $\alpha'\in [0,\alpha)$, any distinct points $x_1, x_2$ in $\Lambda_N$ and $t>0$,
\begin{align*}
\dlim_{N\to \infty} &\Big|\mathbf{E}^{\eps}\big[\eta(x_1, N^{2+\alpha'}t)\,\eta(x_2, N^{2+\alpha'}t)\big]\\
&-\mathbf{E}^{\eps}[\eta(x_1,N^{2+\alpha'}t)]\,\mathbf{E}^{\eps}[\eta(x_2,N^{2+\alpha'}t)]\Big|=0.
\end{align*}
\end{thm}
\begin{proof}
It is well-known that we have the propagation of chaos property for the simple symmetric exclusion process in $\{0,1\}^{\Lambda_N}$ with reservoirs at the boundaries of fixed densities $v_{0,\pm}$. Therefore, in view of this result and \eqref{f}, in order to obtain Theorem \ref{POCsm}, it suffices to prove that for any distinct points $x_1, x_2$ in $\Lambda_N$ and $T>0$,
\begin{align}\label{enough}
\dlim_{N\to\infty}\sup_{t\in [0,T]}&\Big|\mathbf{E}^{\eps}\big[\eta(x_1, N^{2+\alpha'}t)\,\eta(x_2, N^{2+\alpha'}t)\big]\notag\\
&\hskip 1cm-\mathcal{E}^{\eps}\big[\eta(x_1, N^{2+\alpha'}t)\,\eta(x_2, N^{2+\alpha'}t)\big]\Big|=0.
\end{align}

Let us consider the process $(X_1(t), X_2(t))_{t\geq 0}$ of two stirring walks sticky at $\{0,N+1\}$. We call $\hat{\tau}_0, \hat{\tau}_{N+1}$ the first time when $(X_1,X_2)$ reaches $0$ and $N+1$, respectively. Let us denote $\hat{\tau}=\hat{\tau}_0\wedge\hat{\tau}_{N+1}$ and
\begin{align*}
\mathbb{F}(s)&=\mathbb{P}_{(x_1,x_2)}(\hat{\tau}_0\leq s, \hat{\tau}_0<\hat{\tau}_{N+1})\\
\mathbb{G}(s)&=\mathbb{P}_{(x_1,x_2)}(\hat{\tau}_{N+1}\leq s, \hat{\tau}_0>\hat{\tau}_{N+1}).
\end{align*} 

Then up to $\hat{\tau}$, we have the following duality
\begin{align*}
&\mathbf{E}^{\eps}\big[\eta(x_1, t)\,\eta(x_2, t)\big]=\mathbb{E}_{(x_1,x_2)}\mathbb{E}_{(x_1,x_2)}\mathbf{E}^{\eps}\big[\eta_{t-\hat{\tau}}(X_1(\hat{\tau}))\eta_{t-\hat{\tau}}(X_2(\hat{\tau}))\big]\\
=&\dint_0^tE_{x_1}\mathbf{E}^{\eps}\Big[\eta_{t-s}(X_1(s))\dfrac{n_-(t-s)}{M}\Big]\,d\mathbb{F}(s)\\
&\hskip 3cm+\dint_0^tE_{x_1}\mathbf{E}^{\eps}\Big[\eta_{t-s}(X_1(s))\dfrac{n_+(t-s)}{M}\Big]\,d\mathbb{G}(s).
\end{align*}

It follows that
\begin{align*}
&\Big|\mathbf{E}^{\eps}\big[\eta(x_1, N^{2+\alpha'}t)\,\eta(x_2, N^{2+\alpha'}t)\big]-\mathcal{E}^{\eps}\big[\eta(x_1, N^{2+\alpha'}t)\,\eta(x_2, N^{2+\alpha'}t)\big]\Big|\\
=&\Big|\dint_0^{N^{2+\alpha'}t}E_{x_1}\mathbf{E}^{\eps}\Big[\eta_{N^{2+\alpha'}t-s}(X_1(s))\Big(\dfrac{n_-(N^{2+\alpha'}t-s)}{M}-v_{0,-}\Big)\Big]\,d\mathbb{F}(s)\\
&\hskip 0.3cm+\dint_0^{N^{2+\alpha'}t}E_{x_1}\mathbf{E}^{\eps}\Big[\eta_{N^{2+\alpha'}(t-s)}(X_1(s))\Big(\dfrac{n_+(N^{2+\alpha'}t-s)}{M}-v_{0,+}\Big)\Big]\,d\mathbb{G}(s)\Big|\\
\leq&\dint_0^t\mathbf{E}^{\eps}\Big[\Big|\dfrac{n_-(N^{2+\alpha'}s)}{M}-v_{0,-}\Big|\Big]\,d\mathbb{F}(N^{2+\alpha'}(t-s))\\
&\hskip 1cm+\dint_0^t\mathbf{E}^{\eps}\Big[\Big|\dfrac{n_+(N^{2+\alpha'}s)}{M}-v_{0,+}\Big|\Big]\,d\mathbb{G}(N^{2+\alpha'}(t-s)).
\end{align*}

In view of Lemma \ref{fixed}, it implies \eqref{enough} and this completes our proof.
\end{proof}
\section{Adiabatic limit}\label{Adiabatic}
For the case when $\alpha'=\alpha$, the adiabatic limit is obtained and its verification can be found in this section.
\begin{thm}[Adiabatic limit]\label{PDE}
Assume that at the initial time, \eqref{e} is satisfied. Then for any $r \in [0,1]$ and $t>0$,
$\rho^\eps([Nr],N^{2+\alpha}t)$ and  $\rho^\eps_{\pm}(N^{2+\alpha}t)$ have limits $u(r,t)$ and respectively $v_{\pm}(t)$, where
$$u(r,t)=\big(v_+(t)-v_-(t)\big)r+v_-(t),$$
with $v_{\pm}(t)$ such that $\dlim_{t\to 0}v_{\pm}(t) = v_{0,\pm}$ and 
$$\begin{cases}
\dfrac{d}{dt}v_-(t)&=\dfrac{1}{2}\big(v_+(t)-v_-(t)\big)\\
\dfrac{d}{dt}v_+(t)&=-\dfrac{1}{2}\big(v_+(t)-v_-(t)\big).
\end{cases}$$
\end{thm}
\begin{proof}
At first, we study the existence of the limit of the sequence $\rho_{\pm}^{\eps}(N^{2+\alpha}t)$. 

Let $Y$ be a simple symmetric random walk on $\mathbb{Z}$ starting from $x$. We denote by $Y^{\rm rf}$ the simple random walk $Y$ reflected at $0$ and $N+1$ and define the local time spent by $Y^{\rm rf}$ at $0$ and $N+1$ by
  \begin{equation*}
 \mathbf{T}(0,N+1;t;Y^{\rm rf}) = \int_0^t \big( \mathbf 1_{Y^{\rm rf}(s)=0}+ \mathbf 1_{Y^{\rm rf}(s)=N+1}\big)\,ds.
 \end{equation*}

Then it is shown in Proposition 3.3, \cite{N}, that the sticky random walk $X$ can be realized by setting
 \begin{equation}\label{realize}
X\Big(t+ (2M-1)\mathbf{T}(0,N+1;t;Y^{\rm rf})\Big)=Y^{\rm rf}(t).
 \end{equation}

This setting enables us to obtain the following estimates.
\begin{prop}\label{uniconv}
There exists a constant $C$ such that for any $s,t\geq 0$,
$$|\rho^{\eps}_-(t)-\rho^{\eps}_-(s)|\leq \dfrac{1}{N^{1+\alpha}}\min(C\sqrt{|t-s|},N)+\dfrac{C}{N^{2+\alpha}}|t-s|$$
$$|\rho^{\eps}_+(t)-\rho^{\eps}_+(s)|\leq \dfrac{1}{N^{1+\alpha}}\min(C\sqrt{|t-s|},N)+\dfrac{C}{N^{2+\alpha}}|t-s|.$$
\end{prop}
\begin{proof}
Making a similar argument as presented in Proposition 3.6, \cite{N}, enables us to obtain the desired result. Here we just indicate some remarks. First, by the action of the generator $L$, we can verify the following equalities
\begin{align*}
p_t(x,y)&=p_t(y,x), \forall x,y\in\Lambda_N\\
p_t(0,N+1)&=p_t(N+1,0)\\
Mp_t(0,x)&=p_t(x,0), \forall x\in\Lambda_N\\
Mp_t(N+1,x)&=p_t(x,N+1), \forall x\in\Lambda_N.
\end{align*}

Then in view of \eqref{duality*} and the above equalities, we can write
\begin{align*}
\rho^{\eps}(0,t) =&\rho^{\eps}(0,0)p_t(0,0)+\dfrac{1}{M}\dsum_{x=1}^N\rho^{\eps}(x,0)p_t(x,0)+\rho^{\eps}(N+1,0)p_t(N+1,0)\\
\leq& \rho^{\eps}(0,0)+\dfrac{1}{M}\dsum_{x=1}^Np_t(x,0)+p_t(N+1,0).
\end{align*}

It follows from the setting \eqref{realize} and Proposition 3.6, [10], that we can attain the estimate 
\begin{equation}\label{numb}
 \dsum_{x=1}^Np_t(x,0)\leq \min(C\sqrt{t},N). 
\end{equation}

The upper bound of $p_t(N+1,0)$ is already derived strictly as in \eqref{num}. 

Then it implies that 
$$\rho^{\eps}_-(t)-\rho^{\eps}_-(0)\leq \dfrac{1}{N^{1+\alpha}}\min(C\sqrt{t},N)+\dfrac{C}{N^{2+\alpha}}t.$$
\end{proof}

In view of Proposition \ref{uniconv}, we deduce that 
$$|\rho^{\eps}_-(N^{2+\alpha}t)-\rho^{\eps}_-(N^{2+\alpha}s)|\leq C\sqrt{|t-s|}+C|t-s|,$$
$$|\rho^{\eps}_+(N^{2+\alpha}t)-\rho^{\eps}_+(N^{2+\alpha}s)|\leq C\sqrt{|t-s|}+C|t-s|.$$

Let us call $\tilde{\rho}_{\pm}^{\eps,\alpha}(t):=\rho^{\eps}_{\pm}(N^{2+\alpha}t)$.
Thus the following consequence holds.
\begin{cor}\label{subseq}
For any $T>0$, there exist subsequences $\tilde{\rho}_{\pm}^{\eps_k, \alpha}$ that converge uniformly to $v_{\pm}$ on $[0,T]$, respectively. Moreover, $v_{\pm}\in C([0,\infty))$.
\end{cor}

Next, we denote $\tilde{\rho}^{\eps, \alpha}(r,t):=\rho^{\eps}([Nr],N^{2+\alpha}t)$. In order to obtain Theorem \ref{PDE}, we aim to prove that the sequence $ \tilde{\rho}^{\eps, \alpha}$ converges pointwise to $u$ on $[0,1]\times (0,\infty)$, where the hydrodynamic limit $u$ is defined by
$$u(r,t)=\big(v_+(t)-v_-(t)\big)r+v_-(t),$$
where
\begin{align*}
v_-(t)&=\dfrac{v_{0,-}+v_{0,+}}{2}+\dfrac{v_{0,-}-v_{0,+}}{2}e^{-t}\\
v_+(t)&=\dfrac{v_{0,-}+v_{0,+}}{2}-\dfrac{v_{0,-}-v_{0,+}}{2}e^{-t}.
\end{align*}

Applying the same argument we used in the previous section for getting \eqref{main1} gives us
\begin{align}\label{main}
&\tilde{\rho}^{\eps, \alpha}(r,t)\notag\\
=&E_x\big[u_0(\eps Y(N^{2+\alpha}t))\mathbf{1}_{\tau_0^Y\wedge\tau_{N+1}^Y>N^{2+\alpha}t}\big]+\dint_0^{N^{2+\alpha}t}\rho^{\eps}_-(N^{2+\alpha}t-s)\,d\mathds{F}_{N,x}(s)\notag\\
&\hskip5.5cm+\dint_0^{N^{2+\alpha}t}\rho^{\eps}_+(N^{2+\alpha}t-s)\,d\mathds{G}_{N,x}(s).
\end{align}

We observe that the first term on the right hand side of \eqref{main} converges to $0$ as $N\to\infty$ by the estimate \eqref{prob1}.

Moreover, \eqref{prob1} also allows us to get that 
\begin{align*}
&\Big|\dint_0^{N^{2+\alpha}t}\rho^{\eps}_-(N^{2+\alpha}t-s)\,d\mathds{F}_{N,x}(s)-\dint_0^{N^{2+\frac{\alpha}{2}}t}\rho^{\eps}_-(N^{2+\alpha}t-s)\,d\mathds{F}_{N,x}(s) \Big|\\
\leq &\mathds{F}_{N,x}(N^{2+\alpha}t)-\mathds{F}_{N,x}(N^{2+\frac{\alpha}{2}}t)\\
\leq &P_x(\tau_0^Y\wedge\tau_{N+1}^Y>N^{2+\frac{\alpha}{2}}t)\\
\leq &\dfrac{Cx}{N^{1+\frac{\alpha}{4}}\sqrt{t}}.
\end{align*}

On the other hand, by Corollary \ref{subseq}, for any $T>0$, there exist subsequences $\tilde{\rho}_{\pm}^{\eps_k,\alpha}$ that converge uniformly on $[0,T]$ to $v_{\pm}$ respectively. Let us verify the following uniform limit
\begin{equation}\label{equi}
\dlim_{k\to\infty}\sup_{t\in [0,T]}\dint_0^{N^{2+\frac{\alpha}{2}}t}\big|\rho^{\eps_k}_-(N^{2+\alpha}t-s)-v_-(t)\big|\,d\mathds{F}_{N,x}(s)=0.
\end{equation}

In view of Proposition \ref{uniconv}, the term in the above limit is bounded by
\begin{align*}
&\sup_{t\in [0,T]}\dint_0^{N^{2+\frac{\alpha}{2}}t}\big|\tilde{\rho}^{\eps_k, \alpha}_-(t-N^{-(2+\alpha)}s)-\tilde{\rho}^{\eps_k, \alpha}_-(t)\big|\,d\mathds{F}_{N,x}(s)\\
&\quad+\sup_{t\in [0,T]}\dint_0^{N^{2+\frac{\alpha}{2}}t}\big|\tilde{\rho}^{\eps_k, \alpha}_-(t)-v_-(t)\big|\,d\mathds{F}_{N,x}(s)\\
\leq &C\sup_{t\in [0,T]}\dint_0^{N^{2+\frac{\alpha}{2}}t}\Big(\dfrac{\sqrt{s}}{N^{1+\frac{\alpha}{2}}}+\dfrac{s}{N^{2+\alpha}}\Big)\,d\mathds{F}_{N,x}(s)+\sup_{t\in [0,T]}\big|\tilde{\rho}^{\eps_k, \alpha}_-(t)-v_-(t)\big|\\
\leq &\dfrac{C}{N^{\frac{\alpha}{4}}}(\sqrt{T}+T)+\sup_{t\in [0,T]}\big|\tilde{\rho}^{\eps_k, \alpha}_-(t)-v_-(t)\big|.
\end{align*}

Thus it implies \eqref{equi}.

Next, using \eqref{prob1} gives us
$$|\mathds{F}_{N,x}(N^{2+\frac{\alpha}{2}}t)-P_x(\tau_0^Y<\tau_{N+1}^Y)|\leq P_x(\tau_0^Y>N^{2+\frac{\alpha}{2}}t)\leq \dfrac{Cx}{N^{1+\frac{\alpha}{4}}\sqrt{t}}.$$

Since
$$ P_x(\tau_0^Y<\tau_{N+1}^Y)=1-\dfrac{x}{N+1}, $$
we attain that for any $(r,t)\in [0,1]\times (0,\infty)$, 
$$ \dlim_{k\to\infty}\dint_0^{N^{2+\alpha}t}\rho^{\eps_k}_-(N^{2+\alpha}t-s)\,d\mathds{F}_{N,x}(s) = v_-(t)(1-r). $$

and similarly,
$$ \dlim_{k\to\infty}\dint_0^{N^{2+\alpha}t}\rho^{\eps_k}_+(N^{2+\alpha}t-s)\,d\mathds{G}_{N,x}(s) = v_+(t)r. $$

Hence, the subsequence $ \tilde{\rho}^{\eps_k, \alpha}$ converges pointwise on $[0,1]\times (0,\infty)$ to $u$, where
\begin{equation}\label{rep}
u(r,t)=\big(v_+(t)-v_-(t)\big)r+v_-(t).
\end{equation}

On the identification of the limit functions $v_{\pm}$, we use the same technique as introduced in Proposition 3.9, \cite{N}. More precisely, we have the following result.
\begin{prop}\label{prop3.8}
For any $l\in [0,1]$ and $t,t_0\in (0,\infty)$, we have
\begin{align*}
&\dint_{t_0}^t \dfrac{1}{2}u_r(l,s)\,ds+v_-(t_0)-v_-(t)=0\\
&\dint_{t_0}^t -\dfrac{1}{2}u_r(l,s)\,ds+v_+(t_0)-v_+(t)=0.
\end{align*}
\end{prop}
\begin{proof}
We call $L=\eps^{-1}l, H=\eps^{-1}h$. Let us consider the average
$$ \psi^{\alpha}_{\eps,h}(l,t):=\dfrac{1}{H}\dsum_{x=L-H+1}^L\,\eps\dsum_{y=1}^x\,\rho^{\eps}(y,N^{2+\alpha}t). $$

We recall the differential system \eqref{2.6}, \eqref{2.7}, then for any $t,t_0>0$, the difference $\eps^{\alpha}(\psi^{\alpha}_{\eps,h}(l,t)-\psi^{\alpha}_{\eps,h}(l,t_0))$ can be written as follows.
\begin{align}
&\eps^{\alpha}(\psi^{\alpha}_{\eps,h}(l,t)-\psi^{\alpha}_{\eps,h}(l,t_0))\notag\\
=&\dfrac{1}{H}\dsum_{x=L-H+1}^L\eps^{1+\alpha}\dint_{N^{2+\alpha}t_0}^{N^{2+\alpha}t}\dsum_{y=1}^x\dfrac{d}{ds}\rho^{\eps}(y,s)\,ds\notag\\
=&\dfrac{1}{H}\dsum_{x=L-H+1}^L\eps^{1+\alpha}\dint_{t_0}^tN^{2+\alpha}\dfrac{1}{2}\big[\rho^{\eps}(x+1,N^{2+\alpha}s)-\rho^{\eps}(x,N^{2+\alpha}s)\notag\\
&\hskip 5cm+\rho^{\eps}(0,N^{2+\alpha}s)-\rho^{\eps}(1,N^{2+\alpha}s)\big]\,ds\notag\\
=&\dint_{t_0}^t\dfrac{1}{2}\dfrac{\rho^{\eps}(L+1,N^{2+\alpha}s)-\rho^{\eps}(L-H+1,N^{2+\alpha}s)}{h}\,ds\notag\\
&\hskip 5cm+\rho^{\eps}_-(N^{2+\alpha}t_0)-\rho^{\eps}_-(N^{2+\alpha}t).\label{laeq}
\end{align}

Since 
$$\psi^{\alpha}_{\eps,h}(l,t)=\dfrac{\eps}{H}\bigg(H\dsum_{x=1}^{L-H+1}\rho^{\eps}(x,N^{2+\alpha}t)+\dsum_{y=L-H+2}^L\,(L+1-y)\rho^{\eps}(y,N^{2+\alpha}t)\bigg)$$
then
$$\Big|\psi^{\alpha}_{\eps,h}(l,t)-\eps\dsum_{x=1}^{L-H+1}\rho^{\eps}(x,N^{2+\alpha}t)\Big|\leq\dfrac{\eps}{H}\dsum_{y=L-H+2}^L\,(L+1-y)=\dfrac{h-\eps}{2}.$$

Along the subsequence $\rho^{\eps_k}$, we have
\begin{equation}\label{H}
\lim_{h\to 0}\lim_{\eps\to 0}\psi^{\alpha}_{\eps,h}(l,t)=\dint_0^lu(r,t)\,dr. 
\end{equation}

Thus
$$ \lim_{h\to 0}\lim_{\eps\to 0}\eps^{\alpha}(\psi^{\alpha}_{\eps,h}(l,t)-\psi^{\alpha}_{\eps,h}(l,t_0))=0. $$

Moreover, taking the limit along the subsequence $\rho^{\eps_k}$ and then $h\to 0$ of the right hand side of the last equality \eqref{laeq} will give us the desired results.
\end{proof}

Proposition \ref{prop3.8} and the representation \eqref{rep} enable us to deduce that 
\begin{align*}
\dfrac{d}{dt}v_-(t)&=\dfrac{1}{2}\big(v_+(t)-v_-(t)\big)\\
\dfrac{d}{dt}v_+(t)&=-\dfrac{1}{2}\big(v_+(t)-v_-(t)\big).
\end{align*}

Since $\dlim_{t\to 0}v_{\pm}(t) = v_{0,\pm}$, it implies that
\begin{align}
v_-(t)&=\dfrac{v_{0,-}+v_{0,+}}{2}+\dfrac{v_{0,-}-v_{0,+}}{2}e^{-t}\label{-}\\
v_+(t)&=\dfrac{v_{0,-}+v_{0,+}}{2}-\dfrac{v_{0,-}-v_{0,+}}{2}e^{-t}.\label{+}
\end{align}

Notice that we have just obtained the convergence of the sequence $\tilde{\rho}^{\eps,\alpha}$ up to a subsequence. Now if we choose other subsequences $\tilde{\rho}^{\eps_m,\alpha}_{\pm}$, making a similar argument as above will give us their limits $\hat{v}_{\pm}$ defined as in \eqref{-} and \eqref{+}. It follows that the corresponding subsequence $\tilde{\rho}^{\eps_m,\alpha}$ also converges to $u$ given by \eqref{rep}, where $v_{\pm}$ satisfy \eqref{-} and \eqref{+}. This completes the proof of our theorem.
\end{proof}
\section{Global equilibrium limit}\label{Global}
\begin{thm}[Global equilibrium limit]\label{PDElar}
Let $\rho^{\eps}_{\pm}(0)=v_{0,\pm}\in [0,1] $. Then for any $\alpha'>\alpha$, $r \in [0,1]$ and $t>0$, the sequences
$\rho^\eps([Nr],N^{2+\alpha'}t)$ and  $\rho^\eps_{\pm}(N^{2+\alpha'}t)$ have the same limit $(v_{0,-}+v_{0,+})/2$.
\end{thm}
\begin{proof} The verification is split into some steps.
\begin{prop}\label{U}For any $\alpha'>\alpha$, 
$$\dlim_{N\to\infty}P_{N+1}(\tau_0^X<N^{2+\alpha'}t)=1.$$
\end{prop}
\begin{proof}
We first observe that for any $x\in\Lambda_N$, 
\begin{equation}\label{coup}
P_x(\tau_0^X\leq t)\geq P_{N+1}(\tau_0^X\leq t).
\end{equation}

Indeed, let us consider the coupling $\underline{Y}=(Y_1, Y_2)$ constructed as follows: $Y_1$ and $Y_2$ are sticky random walks whose generators are given by \eqref{gen}, they start from $x$ and $N+1$, respectively, move independently up to the first time when they meet each other and from that time they move in the same way. Then we can write
\begin{align*}
P_x(\tau_0^X\leq t)-P_{N+1}(\tau_0^X\leq t)=&\mathbb{E}_{(x,N+1)}\big[\mathbf{1}_{\tau_0^{Y_1}\leq t}-\mathbf{1}_{\tau_0^{Y_2}\leq t}\big]\\
=&\mathbb{E}_{(x,N+1)}\Big[\big[\mathbf{1}_{\tau_0^{Y_1}\leq t}-\mathbf{1}_{\tau_0^{Y_2}\leq t}\big]\mathbf{1}_{\tau_0^{Y_1}<\tau_0^{Y_2}}\Big]\geq 0.
\end{align*}

Next, we will check that there exists $\delta>0$ such that
\begin{equation}\label{del}
P_{N+1}(\tau_0^X<2N^{2+\alpha})\geq \delta.
\end{equation}

Let $\bar{X}$ be a sticky random walk moving on $\bar{\Lambda}_N$, starting from $N+1$, with jump rates $\bar{c}(x,x\pm 1)=\dfrac{1}{2}, \forall x\in\Lambda_N$ and $\bar{c}(0,1)=\bar{c}(N+1,N)=\dfrac{1}{2N}$. We consider a random walk $\bar{Y}$ on $\mathbb{Z}$ sticky at $N+1$ moving in the same way as $\bar{X}$ up to the first time when $\bar{X}$ reaches $0$. Then
\begin{align*}
P_{N+1}(\tau_0^X<2N^{2+\alpha})\geq & P_{N+1}(\tau_0^{\bar{X}}<N^2)= P_{N+1}(\tau_0^{\bar{Y}}<N^2)\\
\geq & P_{N+1}(\bar{Y}(N^2)<0).
\end{align*}

Based on the proof presented in \cite{A}, it can be shown that the sticky random walk $\bar{Y}$ converges uniformly almost surely on compact intervals of $[0,\infty)$ to the Brownian motion $\bar{B}$ sticky at $1$ with sticky coefficient $1/2$. Moreover, by Proposition 11, Section 2.4, \cite{H}, the transition probability kernel for $\bar{B}$ is given by
\begin{align*}
\bar{p}_t(x,y)=&\dfrac{1}{\sqrt{2\pi t}}\Big(e^{-\frac{(x-y+1)^2}{2t}}-e^{-\frac{(|x|+|y-1|)^2}{2t}}\Big)\\
&+\dfrac{1}{2}e^{|x|+|y-1|}e^{t/2}\Erfc\Big(\dfrac{\sqrt{2t}}{2}+\dfrac{|x|+|y-1|}{\sqrt{2t}}\Big)\\
&+\delta_1(y)e^{|x|}e^{t/2}\Erfc\Big(\dfrac{\sqrt{2t}}{2}+\dfrac{|x|}{\sqrt{2t}}\Big),
\end{align*}
where $\Erfc(x)=\dint_x^{\infty}\dfrac{2}{\sqrt{\pi}}e^{-z^2}\,dz$. Hence,
\begin{align*}
 \dlim_{N\to\infty}P_{N+1}(\bar{Y}(N^2)<0)=&P_1(\bar{B}(1)<0)=\dint_{-\infty}^0\bar{p}_1(1,y)\,dy\\
=&\dint_{-\infty}^0\dfrac{1}{2}e^{2-y}e^{1/2}\Erfc\Big(\dfrac{3-y}{\sqrt{2}}\Big)\,dy>0.
\end{align*}

It implies \eqref{del}.

Now we set $\mathcal{T}=2N^{2+\alpha}$ and $K=\left[\dfrac{N^{2+\alpha'}t}{\mathcal{T}}\right]$. Using the Markov property, the estimates \eqref{coup} and \eqref{del} allows us to obtain
\begin{align*}
P_{N+1}(\tau_0^X>2\mathcal{T})=&P_{N+1}(\tau_0^X>\mathcal{T}, \tau_0^X>2\mathcal{T})\\
=&E_{N+1}\Big[E_{N+1}\big[\mathbf{1}_{\tau_0^X>\mathcal{T}}\mathbf{1}_{\tau_0^X>2\mathcal{T}}|\mathcal{F}_{\mathcal{T}}\big]\Big]\\
=&E_{N+1}\Big[\mathbf{1}_{\tau_0^X>\mathcal{T}}E_{X_{\mathcal{T}}}\big[\mathbf{1}_{\tau_0^X>\mathcal{T}}\big]\Big]\\
\leq&(1-\delta)P_{N+1}(\tau_0^X>\mathcal{T})\\
\leq& (1-\delta)^2.
\end{align*}

Then it can be deduced that
$$P_{N+1}(\tau_0^X>N^{2+\alpha'}t)\leq P_{N+1}(\tau_0^X>K\mathcal{T})\leq (1-\delta)^K.$$

Since $\alpha'>\alpha$, we have $K\to\infty$ when $N\to\infty$. Hence, the above estimate gives us the desired result.
\end{proof}
\begin{lem}\label{equal}
For any $\alpha'>\alpha$ and $t>0$,
$$ \dlim_{N\to\infty}\big|\rho^{\eps}_-(N^{2+\alpha'}t)-\rho^{\eps}_+(N^{2+\alpha'}t)\big|=0, $$
$$ \dlim_{N\to\infty}\big|\rho^{\eps}(x,N^{2+\alpha'}t)-\rho^{\eps}_+(N^{2+\alpha'}t)\big|=0, \forall x\in\Lambda_N. $$
\end{lem}
\begin{proof}
Since both equalities can be verified in the same way, we present here the proof of the first equality. By Proposition \ref{prop3.1}, it remains to check that
$$ \dlim_{N\to\infty}\big|E_0\big[\rho^{\eps}_0(X(N^{2+\alpha'}t))\big]-E_{N+1}\big[\rho^{\eps}_0(X(N^{2+\alpha'}t))\big]\big|=0. $$

Let us consider the coupling $\underline{Z}=(Z_1, Z_2)$ constructed as follows: $Z_1$ and $Z_2$ are sticky random walks whose generators are given by \eqref{gen}, they start from $0$ and $N+1$, respectively, move independently up to the first time when they meet each other and from that time they move in the same way. Hence, the difference in the above limit can be rewritten as
\begin{align*}
&\Big|\mathbb{E}_{(0,N+1)}\Big[\rho^{\eps}_0(Z_1(N^{2+\alpha'}t))-\rho^{\eps}_0(Z_2(N^{2+\alpha'}t))\Big]\Big|\\
=&\Big|\mathbb{E}_{(0,N+1)}\Big[[\rho^{\eps}_0(Z_1(N^{2+\alpha'}t))-\rho^{\eps}_0(Z_2(N^{2+\alpha'}t))]\mathbf{1}_{\mathds{U}^c}\Big]\Big|\\
\leq &\mathbb{P}_{(0,N+1)}(\mathds{U}^c),
\end{align*}
where $\mathds{U}=\{\tau_0^{Z_2}<N^{2+\alpha'}t\}$.
It implies our desired equality due to Proposition \ref{U}.
\end{proof}

Let us go back to the proof of Theorem \ref{PDElar}. Making a similar argument as presented in Proposition 3.4, \cite{N}, gives us the following result.
\begin{cor}[Conservation of mass]\label{conmass} For any $t\geq 0$,
$$ \dsum_{x=1}^N\rho^{\eps}_0(x)+N^{1+\alpha}\big(\rho_0^{\eps}(0)+\rho^{\eps}_0(N+1)\big)=  \dsum_{x=1}^N\rho^{\eps}(x,t)+N^{1+\alpha}\big(\rho^{\eps}(0,t)+\rho^{\eps}(N+1,t)\big).$$
\end{cor}

As a consequence of the conservation of mass obtained from Corollary \ref{conmass}, for any $t\geq 0$, we have
$$ \dsum_{x=1}^N\rho^{\eps}_0(x)+M\big(\rho_-^{\eps}(0)+\rho^{\eps}_+(0)\big)= \dsum_{x=1}^N\rho^{\eps}(x,N^{2+\alpha'}t)+M\big(\rho^{\eps}_-(N^{2+\alpha'}t)+\rho^{\eps}_+(N^{2+\alpha'}t)\big).$$

Then
\begin{align*}
&\bigg|\dfrac{\rho^{\eps}_-(N^{2+\alpha'}t)+\rho^{\eps}_+(N^{2+\alpha'}t)}{2}-\dfrac{v_{0,-}+v_{0,+}}{2}\bigg|\\
=&\dfrac{1}{2N^{1+\alpha}}\bigg|\dsum_{x=1}^N\rho^{\eps}(x,0)-\dsum_{x=1}^N\rho^{\eps}(x,N^{2+\alpha'}t)\bigg|\\
\leq &\dfrac{1}{N^{\alpha}}.
\end{align*}

It follows that
$$ \dlim_{N\to\infty}\dfrac{\rho^{\eps}_-(N^{2+\alpha'}t)+\rho^{\eps}_+(N^{2+\alpha'}t)}{2}=\dfrac{v_{0,-}+v_{0,+}}{2}. $$

Moreover, as a consequence of Lemma \ref{equal}, we can deduce that for any $\alpha'>\alpha, x\in\Lambda_N$ and $t>0$,
$$  \dlim_{N\to\infty}\rho^{\eps}(x,N^{2+\alpha'}t)= \dlim_{N\to\infty}\rho^{\eps}_-(N^{2+\alpha'}t)=\dlim_{N\to\infty}\rho^{\eps}_+(N^{2+\alpha'}t).  $$

Therefore, for any $\alpha'>\alpha$, the sequences $\rho^\eps([Nr],N^{2+\alpha'}t)$ and $\rho^\eps_{\pm}(N^{2+\alpha'}t)$ converge pointwise to the same limit $(v_{0,-}+v_{0,+})/2$.
\end{proof}
{\bf Acknowledgements.} I am truly grateful to Prof. Errico Presutti for providing me with so many worthwhile suggestions and helpful advices throughout discussion sessions.
\bibliographystyle{amsalpha}

\end{document}